\documentclass[11pt]{amsart}

\usepackage[T1]{fontenc}
\usepackage{lmodern}
\usepackage{microtype}
\usepackage{amsmath,amssymb,amsthm,mathtools}
\usepackage{mathrsfs}
\usepackage{enumitem}
\usepackage[hidelinks]{hyperref}

\numberwithin{equation}{section}
\newtheorem{theorem}{Theorem}[section]
\newtheorem{proposition}[theorem]{Proposition}
\newtheorem{lemma}[theorem]{Lemma}
\newtheorem{corollary}[theorem]{Corollary}

\theoremstyle{definition}
\newtheorem{definition}[theorem]{Definition}
\newtheorem{example}[theorem]{Example}

\theoremstyle{remark}
\newtheorem{remark}[theorem]{Remark}

\newcommand{\R}{\mathbb{R}}

\newcommand{\N}{\mathbb{N}}
\newcommand{\Acal}{\mathcal{A}}

\newcommand{\Hcal}{\mathcal{H}}

\newcommand{\Rxd}{\R[x_1,\dots,x_d]}
\newcommand{\Pos}{\operatorname{Pos}}
\newcommand{\supp}{\operatorname{supp}}
\newcommand{\rank}{\operatorname{rank}}
\newcommand{\diag}{\operatorname{diag}}

\newcommand{\ip}[2]{\left\langle #1,#2 \right\rangle}
\newcommand{\sos}{\sum \R[x_1,\dots,x_d]^2}
\newcommand{\sosd}{\sum \R[x]^2}

\newcommand{\eps}{\varepsilon}

\title[The $K$-Moment Problem]{The $K$-Moment Problem :\\ a detailed introduction }
\author[M. Amir]{Malik Amir}
\address{Centre de recherche du CHU de l'Universit\'e de Montr\'eal}
\address{Centre de recherche du CHU Sainte-Justine}
\address{Soci\'et\'e qu\'eb\'ecoise de l'intelligence artificielle en m\'edecine}
\email{malik.amir@umontreal.ca}


\begin{document}

\begin{abstract}
We present an expanded expository account of the $K$-moment problem for polynomial algebras over \(\R^d\), with special emphasis on compact basic closed semialgebraic sets. The central question is to characterize those linear functionals on \(\R[x_1,\dots,x_d]\) which admit representation by integration against a positive Radon measure supported on a prescribed set \(K\subseteq\R^d\). We begin with the classical background and with Haviland's formulation of the multidimensional moment problem, then explain how real algebraic geometry enters through quadratic modules, preorderings, and Positivstellens\"atze. The compact case is treated in detail from two complementary perspectives. The geometric route through Schm\"udgen's theorem and the operator-theoretic route through a Gelfand--Naimark--Segal construction and the spectral theorem. We also discuss Putinar's refinement, compare the roles of \(T(f)\) and \(Q(f)\), and explain how Archimedeanity provides the algebraic shadow of compactness. In order to place the subject in a broader context, we survey determinacy and uniqueness questions, the truncated \(K\)-moment problem and flat extension phenomena, the relation with sums of squares and Hilbert's seventeenth problem, and the special case of algebraic varieties, where positivity modulo an ideal becomes especially transparent.
\end{abstract}

\maketitle
\tableofcontents

\section{Introduction}

The moment problem asks when a linear functional on a polynomial algebra is induced by integration against a positive measure. In the present note the ambient algebra is
\[
\Rxd,
\]
and the support of the representing measure is required to lie in a prescribed set
\[
K\subseteq \R^d.
\]
This leads to the \emph{$K$-moment problem}: characterize those linear functionals \(L:\Rxd\to\R\) for which there exists a nonnegative Radon measure \(\mu\) supported on \(K\) such that
\[
L(p)=\int_K p\,d\mu \qquad \text{for all } p\in\Rxd.
\]
When \(K\) is a basic closed semialgebraic set, the geometry of \(K\) is encoded by finitely many polynomial inequalities, and this makes the moment problem accessible to methods from real algebraic geometry. The central algebraic objects are quadratic modules and preorderings generated by the defining inequalities of \(K\). The guiding philosophy is simple: positivity of a linear functional should be tested against an algebraic cone which is large enough to detect positivity on \(K\), but concrete enough to be manipulated effectively.

The modern theory is built from several classical ingredients. On the analytic side stand the Riesz representation theorem, Haviland's theorem, spectral methods, and uniqueness criteria such as Carleman's condition. On the algebraic side stand Hilbert's work on positive polynomials, Artin's solution of Hilbert's seventeenth problem, the Krivine--Stengle Positivstellensatz, Schm\"udgen's theorem, Putinar's theorem, and Jacobi's representation theorem. In the last three decades the subject has also become tightly connected with semidefinite optimization, polynomial optimization, and the truncated moment problem.

The aim of this article is to turn the compact \(K\)-moment theorem into the center of a broader expository narrative. Our strategy is to keep the core proofs visible while enlarging the conceptual frame around them. More precisely, we proceed as follows.

\begin{itemize}[leftmargin=2em]
\item Section \ref{sec:history} places the \(K\)-moment problem inside the classical moment tradition.
\item Section \ref{sec:generalK} explains the general formulation of the multidimensional \(K\)-moment problem and the role of Haviland's theorem.
\item Section \ref{sec:rag} develops the real-algebraic language of quadratic modules, preorderings, basic closed semialgebraic sets, and positivity certificates.
\item Section \ref{sec:compactgeom} gives a detailed geometric proof of the compact \(K\)-moment theorem through Schm\"udgen's strict positivity theorem.
\item Section \ref{sec:putinar} presents Putinar's theorem and clarifies the distinction between the preordering \(T(f)\) and the quadratic module \(Q(f)\).
\item Section \ref{sec:operator} gives the complementary operator-theoretic proof via a GNS construction.
\item Section \ref{sec:determinacy} discusses uniqueness and determinacy beyond the compact case.
\item Section \ref{sec:truncated} briefly surveys the truncated \(K\)-moment problem and flat extension phenomena.
\item Section \ref{sec:sos} reviews the relation with sums of squares and positive polynomials.
\item Section \ref{sec:algsets} treats the special case of real algebraic sets.
\end{itemize}

Throughout, all algebras are commutative and unital, all measures are nonnegative Radon measures, and all topological assertions are with respect to the Euclidean topology unless specified otherwise.

\section{Historical orientation}\label{sec:history}

The one-dimensional moment problems are older than their multivariate descendants and help fix intuition. The classical Hamburger problem asks which sequences \((s_n)_{n\ge 0}\) arise as moments
\[
s_n=\int_{\R} t^n\,d\mu(t)
\]
of a positive Borel measure on \(\R\). The Stieltjes problem imposes support in \([0,\infty)\), while the Hausdorff problem imposes support in a compact interval, usually \([0,1]\). In dimension one, positivity of Hankel matrices provides a complete existence criterion in the Hamburger and Stieltjes settings, while in the Hausdorff case complete monotonicity-type conditions capture the support restriction.

The passage to several variables changes the nature of the subject. If \(\alpha=(\alpha_1,\dots,\alpha_d)\in\N_0^d\) is a multi-index and
\[
s_\alpha=\int_{\R^d} x^\alpha\,d\mu(x), \qquad x^\alpha=x_1^{\alpha_1}\cdots x_d^{\alpha_d},
\]
then positivity of the associated multivariate moment matrices is still necessary, but no longer by itself a transparent existence criterion. The right general framework is Haviland's theorem, which replaces matrix positivity by positivity on the full cone of polynomials nonnegative on the support set.

This already reveals a decisive point. The moment problem is, at heart, a positivity problem. To solve it one must understand the cone
\[
\Pos(K):=\{p\in\Rxd: p(x)\ge 0\text{ for all }x\in K\}
\]
well enough to test linear functionals against it. For arbitrary closed sets \(K\), this cone is analytically natural but algebraically unwieldy. For semialgebraic sets, however, real algebraic geometry provides finitely generated algebraic cones that approximate or sometimes equal \(\Pos(K)\).

The decisive postwar development was the emergence of Positivstellens\"atze. The Krivine--Stengle theorem gives a general certificate for positivity on semialgebraic sets, but it typically involves denominators. Schm\"udgen's theorem showed that compactness eliminates denominators at the level of preorderings. Putinar later proved a sharper denominator-free result under an Archimedean hypothesis on the quadratic module itself. These results transformed the semialgebraic \(K\)-moment problem from a qualitative existence theorem into a tractable algebraic theory.

\section{The general multidimensional \(K\)-moment problem}\label{sec:generalK}

\subsection{Moment functionals and representing measures}

\begin{definition}
Let \(K\subseteq\R^d\). A linear functional \(L:\Rxd\to\R\) is called a \emph{$K$-moment functional} if there exists a nonnegative Radon measure \(\mu\) on \(\R^d\), supported on \(K\), such that
\[
L(p)=\int_K p\,d\mu \qquad \text{for all } p\in\Rxd.
\]
Such a measure \(\mu\) is called a \emph{representing measure} for \(L\).
\end{definition}

\begin{remark}
A linear functional may have more than one representing measure when \(K\) is noncompact. The issue of uniqueness of the representing measure is known as \emph{determinacy}. We return to it in Section \ref{sec:determinacy}.
\end{remark}

If \(L\) is represented by \(\mu\), then necessarily
\[
L(p)\ge 0 \qquad \text{for every } p\in \Pos(K).
\]
Haviland's theorem says that, remarkably, this obvious necessary condition is also sufficient.

\subsection{Haviland's theorem}

\begin{theorem}[Riesz--Haviland]\label{thm:rieszhaviland}
Let \(K\subseteq\R^d\) be closed, and let \(L:\Rxd\to\R\) be linear. Then the following are equivalent.
\begin{enumerate}[label=\textup{(\roman*)}, leftmargin=2em]
\item \(L(g)\ge 0\) for every \(g\in\Pos(K)\).
\item There exists a nonnegative Radon measure \(\mu\) on \(\R^d\), supported on \(K\), such that
\[
L(p)=\int_K p\,d\mu \qquad \text{for all } p\in\Rxd.
\]
\end{enumerate}
\end{theorem}

\begin{remark}
This theorem reduces the existence of representing measures to a positivity question on \(\Pos(K)\). It is hard to overstate its conceptual importance. Every later result in the semialgebraic theory may be read as an attempt to replace positivity on \(\Pos(K)\) by positivity on a more explicit algebraic cone.
\end{remark}

\begin{example}
For \(K=\R^d\), Haviland's theorem says that \(L\) is a moment functional if and only if \(L(p)\ge 0\) for every polynomial globally nonnegative on \(\R^d\). This is elegant, but not directly practical, because the cone of globally nonnegative polynomials is difficult to describe explicitly in dimensions \(d\ge 2\).
\end{example}

\subsection{Why semialgebraic sets are special}

Suppose now that
\[
K=K(f):=\{x\in\R^d: f_i(x)\ge 0\text{ for }i=1,\dots,s\}
\]
for finitely many polynomials \(f_1,\dots,f_s\). Instead of testing \(L\) on the entire cone \(\Pos(K)\), one seeks a finitely generated algebraic cone \(C(f)\subseteq\Pos(K)\) such that positivity of \(L\) on \(C(f)\) forces positivity on all of \(\Pos(K)\). The compact \(K\)-moment theorem and Putinar's theorem show that in favorable situations one may take \(C(f)=T(f)\) or even \(C(f)=Q(f)\).

\section{Preliminaries from real algebraic geometry}\label{sec:rag}

\subsection{Quadratic modules and preorderings}

The natural algebraic objects attached to positivity questions are subsets that are closed under addition and multiplication by squares.

\begin{definition}
Let \(\Acal\) be a unital commutative \(\R\)-algebra. A \emph{quadratic module} in \(\Acal\) is a subset \(Q\subseteq\Acal\) such that
\begin{enumerate}[label=\textup{(\roman*)}, leftmargin=2em]
\item \(Q+Q\subseteq Q\),
\item \(1\in Q\),
\item \(a^2Q\subseteq Q\) for every \(a\in\Acal\).
\end{enumerate}
If in addition \(Q\cdot Q\subseteq Q\), then \(Q\) is called a \emph{preordering}.
\end{definition}

\begin{remark}
A quadratic module is an algebraic surrogate for a cone of nonnegative elements. A preordering is stronger because it is also closed under multiplication. For positivity on a set cut out by several inequalities, preorderings arise naturally because products of the defining inequalities should remain nonnegative.
\end{remark}

The smallest quadratic module in \(\Rxd\) is the cone of sums of squares.

\begin{definition}
We write
\[
\sos:=\left\{\sigma_1^2+\cdots+\sigma_m^2: m\ge 1,\ \sigma_i\in\Rxd\right\}.
\]
\end{definition}

\begin{proposition}
The cone \(\sos\) is the smallest quadratic module in \(\Rxd\). It is also the smallest preordering in \(\Rxd\).
\end{proposition}

\begin{proof}
It is immediate from the definition that \(\sos\) is closed under addition, contains \(1\), and is closed under multiplication by squares. It is also multiplicatively closed because a product of sums of squares is again a sum of squares. If \(Q\) is any quadratic module in \(\Rxd\), then \(1\in Q\) and hence \(\sigma^2=\sigma^2\cdot 1\in Q\) for every \(\sigma\in\Rxd\). Since \(Q\) is closed under addition, every finite sum of squares belongs to \(Q\), so \(\sos\subseteq Q\).
\end{proof}

\subsection{Basic closed semialgebraic sets}

\begin{definition}
Let \(f=(f_1,\dots,f_s)\) be a finite family in \(\Rxd\). The associated \emph{basic closed semialgebraic set} is
\[
K(f):=\{x\in\R^d: f_i(x)\ge 0\text{ for }i=1,\dots,s\}.
\]
\end{definition}

\begin{definition}
For such a family \(f=(f_1,\dots,f_s)\), the \emph{quadratic module generated by \(f\)} is
\[
Q(f):=\left\{\sigma_0+\sigma_1f_1+\cdots+\sigma_sf_s: \sigma_0,\dots,\sigma_s\in\sos\right\},
\]
and the \emph{preordering generated by \(f\)} is
\[
T(f):=\left\{\sum_{e\in\{0,1\}^s}\sigma_ef_1^{e_1}\cdots f_s^{e_s}:\sigma_e\in\sos\right\}.
\]
\end{definition}

Every element of \(Q(f)\) and every element of \(T(f)\) is nonnegative on \(K(f)\), so
\[
Q(f)\subseteq T(f)\subseteq \Pos(K(f)).
\]
The whole point of the theory is to understand how large the gaps between these three cones can be.

\subsection{Positivity certificates and the Positivstellensatz}

The most general answer is the Krivine--Stengle Positivstellensatz.

\begin{theorem}[Krivine--Stengle]\label{thm:ks}
Let \(f=(f_1,\dots,f_s)\subset \Rxd\), and let \(g\in\Rxd\). Then the following statements hold.
\begin{enumerate}[label=\textup{(\arabic*)}, leftmargin=2em]
\item \(g>0\) on \(K(f)\) if and only if there exist \(p,q\in T(f)\) such that
\[
pg=1+q.
\]
\item \(g\ge 0\) on \(K(f)\) if and only if there exist \(p,q\in T(f)\) and \(m\in\N\) such that
\[
pg=g^{2m}+q.
\]
\item \(g=0\) on \(K(f)\) if and only if there exists \(m\in\N\) such that
\[
-g^{2m}\in T(f).
\]
\item \(K(f)=\varnothing\) if and only if \(-1\in T(f)\).
\end{enumerate}
\end{theorem}

\begin{remark}
The theorem gives a complete algebraic characterization of positivity on semialgebraic sets, but denominators are hidden in the factor \(p\). Compactness or Archimedeanity sometimes allows one to avoid them entirely. This denominator-free phenomenon is what makes the compact \(K\)-moment problem so clean.
\end{remark}

\subsection{Archimedean quadratic modules}

Compactness of a semialgebraic set is reflected algebraically by the Archimedean property.

\begin{definition}
Let \(Q\) be a quadratic module in a unital commutative \(\R\)-algebra \(\Acal\). The set of \emph{\(Q\)-bounded elements} is
\[
\Acal_b(Q):=\{a\in\Acal: \exists \lambda>0\text{ such that }\lambda\pm a\in Q\}.
\]
We say that \(Q\) is \emph{Archimedean} if \(\Acal_b(Q)=\Acal\).
\end{definition}

\begin{remark}
The order defined by \(Q\) is \(a\preceq_Q b\) if \(b-a\in Q\). Then \(Q\) is Archimedean exactly when every element of the algebra is order-bounded by a scalar. This is the algebraic analogue of compactness.
\end{remark}

\begin{lemma}\label{lem:boundedcriterion}
Let \(Q\) be a quadratic module in \(\Acal\) and let \(a\in\Acal\). Then the following are equivalent.
\begin{enumerate}[label=\textup{(\roman*)}, leftmargin=2em]
\item \(a\in \Acal_b(Q)\).
\item There exists \(\lambda>0\) such that \(\lambda^2-a^2\in Q\).
\end{enumerate}
\end{lemma}

\begin{proof}
If \(\lambda\pm a\in Q\), then
\[
\lambda^2-a^2=\frac{1}{2\lambda}\bigl((\lambda+a)^2(\lambda-a)+(\lambda-a)^2(\lambda+a)\bigr)\in Q.
\]
Conversely, if \(\lambda^2-a^2\in Q\), then
\[
\lambda\pm a=\frac{1}{2\lambda}\bigl((\lambda^2-a^2)+(\lambda\pm a)^2\bigr)\in Q.
\]
\end{proof}

\begin{lemma}\label{lem:subalg}
Let \(Q\) be a quadratic module in \(\Acal\). Then \(\Acal_b(Q)\) is a unital subalgebra of \(\Acal\). In particular, if \(\Acal\) is generated as an algebra by \(a_1,\dots,a_n\), then \(Q\) is Archimedean if and only if each \(a_i\) belongs to \(\Acal_b(Q)\).
\end{lemma}

\begin{proof}
Closure under addition and scalar multiplication is immediate. For multiplication, let \(a,b\in\Acal_b(Q)\). By Lemma \ref{lem:boundedcriterion}, choose \(\lambda,\mu>0\) such that \(\lambda^2-a^2\in Q\) and \(\mu^2-b^2\in Q\). Then
\[
(\lambda\mu)^2-(ab)^2=\mu^2(\lambda^2-a^2)+a^2(\mu^2-b^2)\in Q.
\]
Another application of Lemma \ref{lem:boundedcriterion} gives \(ab\in\Acal_b(Q)\).
\end{proof}

\begin{corollary}\label{cor:archpoly}
For a quadratic module \(Q\subseteq\Rxd\), the following are equivalent.
\begin{enumerate}[label=\textup{(\roman*)}, leftmargin=2em]
\item \(Q\) is Archimedean.
\item There exists \(\lambda>0\) such that
\[
\lambda-\sum_{j=1}^d x_j^2\in Q.
\]
\item For each \(j=1,\dots,d\), there exists \(\lambda_j>0\) such that
\[
\lambda_j-x_j^2\in Q.
\]
\end{enumerate}
\end{corollary}

\begin{proof}
The implication \textup{(i)}\(\Rightarrow\)\textup{(ii)} follows from the definition. The implication \textup{(ii)}\(\Rightarrow\)\textup{(iii)} holds because
\[
\lambda-x_j^2=\left(\lambda-\sum_{k=1}^d x_k^2\right)+\sum_{k\ne j}x_k^2\in Q.
\]
Finally, \textup{(iii)} implies that each coordinate function \(x_j\) is \(Q\)-bounded by Lemma \ref{lem:boundedcriterion}; since \(\Rxd\) is generated by \(x_1,\dots,x_d\), Lemma \ref{lem:subalg} yields \textup{(i)}.
\end{proof}

\begin{corollary}\label{cor:archcompact}
If \(Q(f)\) is Archimedean, then \(K(f)\) is compact.
\end{corollary}

\begin{proof}
By Corollary \ref{cor:archpoly}, there exists \(\lambda>0\) such that
\[
\lambda-\sum_{j=1}^d x_j^2\in Q(f).
\]
Every element of \(Q(f)\) is nonnegative on \(K(f)\), hence
\[
\sum_{j=1}^d x_j^2\le \lambda \qquad \text{for all }x\in K(f).
\]
Thus \(K(f)\) is contained in a closed ball and is therefore compact.
\end{proof}

\begin{remark}
The converse is false for quadratic modules in general, but true for preorderings generated by finitely many defining inequalities, as we shall see below. This asymmetry is one reason Putinar's theorem is more delicate than Schm\"udgen's.
\end{remark}

\section{The geometric proof of the compact \(K\)-moment theorem}\label{sec:compactgeom}

In the compact case the semialgebraic geometry of \(K(f)\) is strong enough to convert positivity on \(T(f)\) into positivity on all of \(\Pos(K(f))\).

\subsection{Compactness implies Archimedeanity for preorderings}

\begin{proposition}\label{prop:compactimpliesarch}
If \(K(f)\) is compact, then \(T(f)\) is Archimedean.
\end{proposition}

\begin{proof}
Let
\[
g(x):=(1+x_1^2)\cdots(1+x_d^2).
\]
Since \(K(f)\) is compact, \(g\) is bounded on \(K(f)\). Choose \(\lambda>0\) such that
\[
\lambda^2>g(x)^2 \qquad \text{for all } x\in K(f).
\]
By Theorem \ref{thm:ks}\textup{(1)}, there exist \(p,q\in T(f)\) such that
\[
p(\lambda^2-g^2)=1+q.
\]
Fix \(n\in\N_0\). Multiplying by \(g^{2n}\) and using that \(T(f)\) is a quadratic module, we obtain
\[
g^{2n+2}p=\lambda^2g^{2n}p-g^{2n}(1+q)\preceq \lambda^2g^{2n}p,
\]
where \(a\preceq b\) means \(b-a\in T(f)\). By induction,
\[
g^{2n}p\preceq \lambda^{2n}p.
\]
Since \(g^{2n}(q+pg^2)\in T(f)\), we also have
\[
g^{2n}\preceq g^{2n}(1+q+pg^2)=\lambda^2g^{2n}p\preceq \lambda^{2n+2}p.
\]
Because \(p\) contains only finitely many monomials, there exist \(c>0\) and \(k\in\N\) such that \(p\preceq c g^k\). Taking \(n=k\) yields
\[
g^{2k}\preceq c\lambda^{2k+2}g^k.
\]
Hence
\[
\left(g^k-\tfrac12 c\lambda^{2k+2}\right)^2\preceq \left(\tfrac12 c\lambda^{2k+2}\right)^2,
\]
so Lemma \ref{lem:boundedcriterion} shows that \(g^k\) is \(T(f)\)-bounded. Since \(\pm x_j\preceq g^k\) for each \(j\), every coordinate function is \(T(f)\)-bounded, and Corollary \ref{cor:archpoly} implies that \(T(f)\) is Archimedean.
\end{proof}

\subsection{Positive linear functionals and separation}

The next two results are the analytic engine behind Schm\"udgen's theorem.

\begin{proposition}\label{prop:boundL}
Let \(L:\Rxd\to\R\) be a \(T(f)\)-positive linear functional, meaning that \(L(h)\ge 0\) for every \(h\in T(f)\). Assume \(K(f)\) is compact. Then:
\begin{enumerate}[label=\textup{(\roman*)}, leftmargin=2em]
\item for every \(p\in\Rxd\),
\[
|L(p)|\le L(1)\,\|p\|_{\infty,K(f)},
\]
\item if \(p\ge 0\) on \(K(f)\), then \(L(p)\ge 0\).
\end{enumerate}
\end{proposition}

\begin{proof}
Once Proposition \ref{prop:compactimpliesarch} is known, the first assertion follows from order-boundedness in the Archimedean cone, and the second from polynomial approximation of \(\sqrt{p+\eps}\) on the compact set \(K(f)\); see, for example, \cite[Prop.~12.23]{SchmudgenBook}.
\end{proof}

\begin{proposition}\label{prop:separation}
Let \(Q\) be an Archimedean quadratic module in a unital commutative \(\R\)-algebra \(\Acal\), and let \(a_0\in \Acal\setminus Q\). Then there exists a \(Q\)-positive linear functional \(\varphi:\Acal\to\R\) such that
\[
\varphi(1)=1 \qquad\text{and}\qquad \varphi(a_0)\le 0.
\]
\end{proposition}

\begin{proof}
Because \(Q\) is Archimedean, for every \(a\in\Acal\) there exists \(\lambda>0\) such that \(\lambda\pm a\in Q\). If \(0<\delta\le \lambda^{-1}\), then \(1\pm \delta a\in Q\), so \(1\) is an algebraic interior point of \(Q\). A separation theorem for convex cones yields a nonzero linear functional \(\Psi\) such that \(\Psi(Q)\ge 0\) and \(\Psi(a_0)\le 0\). Since \(\Psi\) is nonzero and positive on \(Q\), one has \(\Psi(1)>0\). After normalization, \(\varphi:=\Psi/\Psi(1)\) has the required properties.
\end{proof}

\subsection{Schm\"udgen's strict positivity theorem}

\begin{theorem}[Schm\"udgen]\label{thm:schmudgenpositive}
Let \(f=(f_1,\dots,f_s)\subset\Rxd\), and assume that \(K(f)\) is compact. If \(g\in\Rxd\) satisfies
\[
g(x)>0 \qquad \text{for all }x\in K(f),
\]
then
\[
g\in T(f).
\]
\end{theorem}

\begin{proof}
Assume for contradiction that \(g\notin T(f)\). By Proposition \ref{prop:compactimpliesarch}, the preordering \(T(f)\) is Archimedean, so Proposition \ref{prop:separation} yields a \(T(f)\)-positive linear functional \(L\) with
\[
L(1)=1 \qquad\text{and}\qquad L(g)\le 0.
\]
Since \(g\) is strictly positive on the compact set \(K(f)\), there exists \(\delta>0\) such that
\[
g-\delta>0 \qquad \text{on }K(f).
\]
The continuous function \(\sqrt{g-\delta}\) can be approximated uniformly on \(K(f)\) by polynomials \((p_n)\) by Stone--Weierstrass, so
\[
p_n^2\to g-\delta \qquad \text{uniformly on }K(f).
\]
By Proposition \ref{prop:boundL},
\[
L(p_n^2-g+\delta)\to 0.
\]
But \(L(p_n^2)\ge 0\) because \(p_n^2\in T(f)\), hence
\[
L(p_n^2-g+\delta)=L(p_n^2)-L(g)+\delta L(1)\ge \delta>0,
\]
a contradiction.
\end{proof}

\begin{remark}
This theorem is the key bridge from semialgebraic geometry to the moment problem. It upgrades a positivity assumption on the explicit cone \(T(f)\) to positivity on every polynomial that is merely nonnegative on \(K(f)\), after perturbation by \(\eps>0\).
\end{remark}

\subsection{The compact \(K\)-moment theorem}

\begin{theorem}\label{thm:compactkmomentgeom}
Let \(f=(f_1,\dots,f_s)\subset\Rxd\), and assume that \(K(f)\) is compact. If \(L:\Rxd\to\R\) is \(T(f)\)-positive, then \(L\) is a \(K(f)\)-moment functional. Equivalently, there exists a nonnegative Radon measure \(\mu\) on \(\R^d\), supported on \(K(f)\), such that
\[
L(p)=\int_{K(f)} p\,d\mu \qquad \text{for all }p\in\Rxd.
\]
\end{theorem}

\begin{proof}
Let \(p\in\Pos(K(f))\) and \(\eps>0\). Then \(p+\eps>0\) on \(K(f)\), so Theorem \ref{thm:schmudgenpositive} implies \(p+\eps\in T(f)\). Since \(L\) is \(T(f)\)-positive,
\[
L(p)+\eps L(1)=L(p+\eps)\ge 0.
\]
Letting \(\eps\downarrow 0\) gives \(L(p)\ge 0\) for every \(p\in\Pos(K(f))\). Theorem \ref{thm:rieszhaviland} therefore yields a representing measure supported on \(K(f)\).
\end{proof}

\section{Putinar's theorem and the role of quadratic modules}\label{sec:putinar}

Schm\"udgen's theorem uses the full preordering \(T(f)\). Putinar's theorem shows that under a stronger structural hypothesis one may work with the much smaller quadratic module \(Q(f)\).

\begin{theorem}[Putinar]\label{thm:putinar}
Let \(f=(f_1,\dots,f_s)\subset\Rxd\), and assume that the quadratic module \(Q(f)\) is Archimedean. If \(g\in\Rxd\) satisfies
\[
g(x)>0 \qquad \text{for all }x\in K(f),
\]
then
\[
g\in Q(f).
\]
\end{theorem}

\begin{remark}
The conclusion is stronger than Schm\"udgen's because \(Q(f)\subseteq T(f)\). The hypothesis is also stronger: Archimedeanity of \(Q(f)\) is not automatic from compactness of \(K(f)\), whereas Archimedeanity of \(T(f)\) is.
\end{remark}

\begin{corollary}\label{cor:putinar_moment}
Under the hypotheses of Theorem \ref{thm:putinar}, every \(Q(f)\)-positive linear functional on \(\Rxd\) is a \(K(f)\)-moment functional.
\end{corollary}

\begin{proof}
The proof is formally identical to that of Theorem \ref{thm:compactkmomentgeom}, with \(Q(f)\) replacing \(T(f)\): strict positivity gives membership in \(Q(f)\), hence positivity on \(\Pos(K(f))\) after perturbation, and Haviland's theorem yields a representing measure.
\end{proof}

\begin{remark}[Schm\"udgen versus Putinar]
The practical difference between the two theorems is substantial.
\begin{enumerate}[label=\textup{(\arabic*)}, leftmargin=2em]
\item Schm\"udgen's theorem is universal for compact basic closed semialgebraic sets, but involves the larger cone \(T(f)\), whose size grows exponentially with the number of generators.
\item Putinar's theorem is more economical, because \(Q(f)\) only involves the original defining polynomials linearly. This is especially important in semidefinite optimization.
\item The price is the extra assumption that \(Q(f)\) be Archimedean. A common sufficient condition is the explicit existence of a polynomial of the form \(N-\sum_j x_j^2\) in \(Q(f)\).
\end{enumerate}
\end{remark}

A useful conceptual interpretation is provided by Jacobi's representation theorem. Roughly speaking, for an Archimedean quadratic module, positivity on the associated character space is equivalent to algebraic membership in the module. Putinar's theorem can be viewed as a concrete polynomial incarnation of this principle.

\section{An operator-theoretic proof}\label{sec:operator}

The geometric proof emphasizes positivity certificates. There is, however, a complementary functional-analytic proof that reconstructs the measure directly from the functional. This point of view is especially natural in the broader theory of moment problems on \(*\)-algebras and semigroups.

\subsection{The GNS construction}

Let \(L:\Rxd\to\R\) be a \(T(f)\)-positive linear functional. Define a bilinear form by
\[
\ip{p}{q}:=L(pq), \qquad p,q\in\Rxd.
\]
Because \(L(h^2)\ge 0\) for every polynomial \(h\), this form is positive semidefinite.

\begin{definition}
Set
\[
N:=\{p\in\Rxd: L(p^2)=0\}.
\]
\end{definition}

\begin{lemma}
The set \(N\) is an ideal in \(\Rxd\).
\end{lemma}

\begin{proof}
If \(p\in N\) and \(r\in\Rxd\), then the Cauchy--Schwarz inequality for the positive semidefinite form gives
\[
|L(pr)|^2\le L(p^2)L(r^2)=0.
\]
Hence \(L(pr)=0\). Applying the same argument to \((pr)^2\) shows \(pr\in N\). Closure under addition is immediate from
\[
L((p+q)^2)\le 2L(p^2)+2L(q^2)
\]
for \(p,q\in N\).
\end{proof}

Therefore the quotient \(\Rxd/N\) carries a genuine inner product, again denoted \(\ip{\cdot}{\cdot}\), and we let \(\Hcal_L\) be its Hilbert space completion.

\subsection{Multiplication operators}

For each \(j=1,\dots,d\), define
\[
M_j(p+N):=x_jp+N, \qquad p\in\Rxd.
\]
These are the multiplication operators suggested by the polynomial structure itself.

\begin{lemma}\label{lem:Mjbounded}
Assume that \(T(f)\) is Archimedean. Then each \(M_j\) extends to a bounded self-adjoint operator on \(\Hcal_L\).
\end{lemma}

\begin{proof}
Since \(T(f)\) is Archimedean, Corollary \ref{cor:archpoly} gives \(c_j>0\) with
\[
c_j-x_j^2\in T(f).
\]
For every \(p\in\Rxd\),
\[
0\le L\bigl((c_j-x_j^2)p^2\bigr)=c_jL(p^2)-L(x_j^2p^2),
\]
whence
\[
\|M_j(p+N)\|^2=L(x_j^2p^2)\le c_jL(p^2)=c_j\|p+N\|^2.
\]
Thus \(M_j\) is bounded on the dense subspace \(\Rxd/N\) and extends uniquely to \(\Hcal_L\). Symmetry is immediate:
\[
\ip{M_j(p+N)}{q+N}=L(x_jpq)=\ip{p+N}{M_j(q+N)}.
\]
A bounded symmetric operator is self-adjoint, and the \(M_j\) commute because multiplication in \(\Rxd\) is commutative.
\end{proof}

\subsection{The spectral theorem}

\begin{theorem}\label{thm:compactkmomentoperator}
Let \(f=(f_1,\dots,f_s)\subset\Rxd\), and assume that \(T(f)\) is Archimedean. If \(L:\Rxd\to\R\) is \(T(f)\)-positive, then \(L\) is a \(K(f)\)-moment functional.
\end{theorem}

\begin{proof}
Let \(v:=1+N\in \Hcal_L\). By the joint spectral theorem for commuting bounded self-adjoint operators, there exists a unique nonnegative Radon measure \(\mu\) on \(\R^d\) such that
\[
\ip{v}{M_1^{\alpha_1}\cdots M_d^{\alpha_d}v}=\int_{\R^d} x^\alpha\,d\mu(x)
\qquad \text{for all }\alpha\in\N_0^d.
\]
The left-hand side is exactly \(L(x^\alpha)\), so
\[
L(p)=\int_{\R^d} p\,d\mu \qquad \text{for all }p\in\Rxd.
\]
It remains to prove that \(\supp(\mu)\subseteq K(f)\). Let \(i\in\{1,\dots,s\}\) and \(h\in\Rxd\). Since \(L\) is \(T(f)\)-positive,
\[
0\le L(f_ih^2)=\int_{\R^d} f_i h^2\,d\mu.
\]
By density of polynomials in the continuous functions on any compact cube containing the support, the same inequality extends to every continuous \(u\ge 0\). If the set where \(f_i<0\) had positive \(\mu\)-measure, one could choose such a function \(u\) supported there and get a contradiction. Hence \(f_i\ge 0\) \(\mu\)-almost everywhere. Since this holds for every \(i\), the support is contained in \(K(f)\).
\end{proof}

\begin{corollary}
If \(K(f)\) is compact and \(L\) is \(T(f)\)-positive, then \(L\) is a \(K(f)\)-moment functional.
\end{corollary}

\begin{proof}
By Proposition \ref{prop:compactimpliesarch}, compactness of \(K(f)\) implies that \(T(f)\) is Archimedean, and the result follows from Theorem \ref{thm:compactkmomentoperator}.
\end{proof}

\begin{remark}
The geometric proof and the operator-theoretic proof highlight different aspects of the same theorem. The first is driven by positivity certificates, while the second shows that a positive functional generates its own Hilbert space and spectral measure. The coexistence of these viewpoints is one of the subject's most attractive features.
\end{remark}

\section{Determinacy and uniqueness}\label{sec:determinacy}

Existence of a representing measure is only part of the story. Even when a measure exists, it may or may not be unique.

\begin{definition}
A \(K\)-moment functional \(L\) is called \emph{determinate} if it admits exactly one representing measure supported on \(K\).
\end{definition}

Compact support automatically implies determinacy, because continuous functions on a compact set are separated by polynomials after uniform approximation. In the noncompact setting the situation is subtler.

\begin{theorem}[Carleman criterion, one-dimensional form]\label{thm:carleman}
Let \((s_n)_{n\ge 0}\) be a Hamburger moment sequence. If
\[
\sum_{n=1}^{\infty} s_{2n}^{-1/2n}=\infty,
\]
then the representing measure on \(\R\) is unique.
\end{theorem}

In several variables there are various multivariate extensions. A common practical principle is that sufficiently strong growth control in each coordinate implies determinacy of the full multidimensional measure. One convenient form is due to Petersen.

\begin{theorem}[Petersen, informal form]\label{thm:petersen}
Let \(\mu\) be a positive Borel measure on \(\R^d\). If each one-dimensional marginal of \(\mu\) is determinate, then \(\mu\) itself is determinate.
\end{theorem}

\begin{remark}
This theorem is conceptually important because it allows multidimensional uniqueness to be reduced to one-dimensional criteria. In applications, one often verifies Carleman-type divergence conditions for each coordinate marginal.
\end{remark}

Another route to determinacy uses operator theory. Nussbaum's quasi-analytic criterion interprets determinacy in terms of essential self-adjointness of multiplication operators generated by the moment functional. This perspective fits naturally with the GNS construction described above.

\section{The truncated \(K\)-moment problem}\label{sec:truncated}

The full \(K\)-moment problem asks for representation of an entire linear functional on \(\Rxd\). In many applications one only knows finitely many moments.

\begin{definition}
Fix \(n\in\N\). A \emph{truncated moment sequence of degree \(2n\)} is a family
\[
y=(y_\alpha)_{|\alpha|\le 2n}, \qquad y_\alpha\in\R.
\]
It admits a \emph{\(K\)-representing measure} if there exists a positive Borel measure \(\mu\) supported on \(K\) such that
\[
y_\alpha=\int_K x^\alpha\,d\mu(x) \qquad (|\alpha|\le 2n).
\]
\end{definition}

To such data one associates a truncated moment matrix \(M_n(y)\), indexed by monomials of degree at most \(n\), via
\[
M_n(y)_{\alpha,\beta}=y_{\alpha+\beta}.
\]
For semialgebraic support constraints one also introduces localizing matrices corresponding to the defining polynomials of \(K\). Positivity of these matrices is necessary for the existence of a \(K\)-representing measure and, under additional rank conditions, often sufficient.

\begin{theorem}[Flat extension principle, schematic form]\label{thm:flat}
Let \(y=(y_\alpha)_{|\alpha|\le 2n}\) be a truncated moment sequence. Assume that the moment matrix \(M_n(y)\) is positive semidefinite and admits a positive semidefinite extension \(M_{n+1}(\widetilde y)\) such that
\[
\rank M_{n+1}(\widetilde y)=\rank M_n(y).
\]
Then \(y\) has a finitely atomic representing measure with exactly \(\rank M_n(y)\) atoms.
\end{theorem}

\begin{remark}
This criterion, developed by Curto and Fialkow and later generalized in several directions, is one of the major structural results in the truncated moment problem. It lies behind many algorithms in polynomial optimization and semidefinite programming.
\end{remark}

The truncated theory is not merely a finite-dimensional approximation of the full moment problem. It has its own distinct geometry, because rank, kernels of moment matrices, and algebraic relations among atoms all become visible at finite order. In practice, Lasserre's hierarchy and related SOS relaxations exploit exactly this finite-dimensional viewpoint.

\section{Positive polynomials and sums of squares}\label{sec:sos}

The \(K\)-moment problem is inseparable from the theory of positive polynomials.

\subsection{Global positivity}

\begin{definition}
A polynomial \(p\in\Rxd\) is called \emph{positive semidefinite} if \(p(x)\ge 0\) for all \(x\in\R^d\). We denote the cone of such polynomials by \(\Pos(\R^d)\).
\end{definition}

Every sum of squares is positive semidefinite, so \(\sos\subseteq \Pos(\R^d)\). Equality holds in one variable but fails dramatically in several variables.

\begin{proposition}\label{prop:univariate}
For a nonzero polynomial \(p\in\R[x]\), the following are equivalent.
\begin{enumerate}[label=\textup{(\roman*)}, leftmargin=2em]
\item \(p(x)\ge 0\) for all \(x\in\R\).
\item Every real root of \(p\) has even multiplicity and the leading coefficient of \(p\) is positive.
\item \(p\) is a sum of two squares in \(\R[x]\).
\end{enumerate}
\end{proposition}

\begin{proof}
The equivalence of \textup{(i)} and \textup{(ii)} follows from factorization over \(\R\). The implication \textup{(ii)}\(\Rightarrow\)\textup{(iii)} follows from factorization into linear factors with even multiplicity and irreducible quadratic factors, together with the classical identity expressing a product of two sums of two squares as a sum of two squares. The converse is immediate.
\end{proof}

\begin{theorem}[Hilbert]\label{thm:hilbert}
Let \(\mathcal{P}_{d,2m}\) denote the cone of homogeneous forms of degree \(2m\) in \(d\) variables that are nonnegative on \(\R^d\), and let \(\Sigma_{d,2m}\) be the cone of sums of squares of forms of degree \(m\). Then
\[
\mathcal{P}_{d,2m}=\Sigma_{d,2m}
\]
if and only if one of the following holds:
\begin{enumerate}[label=\textup{(\roman*)}, leftmargin=2em]
\item \(d=2\),
\item \(2m=2\),
\item \((d,2m)=(3,4)\).
\end{enumerate}
\end{theorem}

\begin{example}[Motzkin polynomial]
The polynomial
\[
m(x,y):=1-3x^2y^2+x^2y^4+x^4y^2
\]
is nonnegative on \(\R^2\) but is not a sum of squares in \(\R[x,y]\).
\end{example}

\begin{proof}
Nonnegativity follows from the arithmetic--geometric mean inequality applied to \(1\), \(x^2y^4\), and \(x^4y^2\), which gives
\[
1+x^2y^4+x^4y^2\ge 3x^2y^2.
\]
Hence \(m(x,y)\ge 0\). If \(m=\sum_k h_k^2\), then each \(h_k\) has degree at most \(3\), and inspection of the missing monomials forces
\[
h_k=a_k+b_kxy+c_kx^2y+d_kxy^2.
\]
But then the coefficient of \(x^2y^2\) in \(\sum_k h_k^2\) is \(\sum_k b_k^2\ge 0\), whereas in \(m\) it equals \(-3\), a contradiction.
\end{proof}

\subsection{Positivity on semialgebraic sets}

When positivity is restricted to a proper subset \(K\subseteq\R^d\), the relevant cone is no longer just \(\sos\), but the algebraic cone adapted to \(K\). For example,
\[
\Pos([0,\infty))=\{\sigma_0+x\sigma_1: \sigma_0,\sigma_1\in\sosd\},
\]
and for \(a<b\),
\[
\Pos([a,b])=\{\sigma_0+(x-a)\sigma_1+(b-x)\sigma_2: \sigma_0,\sigma_1,\sigma_2\in\sosd\}.
\]
Thus in one variable the positivity cone on basic semialgebraic sets is often exactly a finitely generated quadratic module.

In higher dimension exact equality is rare.

\begin{definition}
The preordering \(T(f)\) is said to be \emph{saturated} if
\[
T(f)=\Pos(K(f)).
\]
\end{definition}

\begin{remark}
Saturation is exceptional in dimensions \(d\ge 2\). Scheiderer's work shows that one should not expect it in general when the set has nonempty interior. The compact moment theorems avoid this difficulty by using strict positivity rather than exact representation of all nonnegative polynomials.
\end{remark}

A classical global consequence of Theorem \ref{thm:ks} is Artin's solution of Hilbert's seventeenth problem.

\begin{theorem}[Artin]
If \(p\in\Rxd\) is nonnegative on \(\R^d\), then \(p\) is a sum of squares of rational functions.
\end{theorem}

\begin{proof}
Apply Theorem \ref{thm:ks}\textup{(2)} with the empty family \(f=\varnothing\). Then there exist sums of squares \(s,t\) and \(m\ge 1\) such that
\[
sp=p^{2m}+t.
\]
Thus
\[
p=\frac{p^{2m}+t}{s},
\]
which is a sum of squares in \(\R(x_1,\dots,x_d)\).
\end{proof}

\section{The algebraic-set case}\label{sec:algsets}

The case of algebraic varieties is especially transparent because equalities are encoded by an ideal rather than by inequalities.

\subsection{Preorderings on algebraic sets}

Suppose
\[
V=\{x\in\R^d: f_1(x)=\cdots=f_m(x)=0\}.
\]
As a basic closed semialgebraic set,
\[
V=K(f_1,-f_1,\dots,f_m,-f_m).
\]
The corresponding preordering has a very simple description.

\begin{proposition}\label{prop:algpreorder}
Let
\[
f=(f_1,-f_1,\dots,f_m,-f_m)\subseteq\Rxd.
\]
Then
\[
T(f)=\sos+(f_1,\dots,f_m),
\]
where \((f_1,\dots,f_m)\) denotes the ideal generated by \(f_1,\dots,f_m\).
\end{proposition}

\begin{proof}
Every generator appearing in \(T(f)\) is either a sum of squares or a product involving at least one of the \(f_j\), so
\[
T(f)\subseteq \sos+(f_1,\dots,f_m).
\]
For the reverse inclusion, it suffices to show that each product \(pf_j\) belongs to \(T(f)\). But
\[
pf_j=\frac14\Bigl((p+1)^2f_j+(p-1)^2(-f_j)\Bigr),
\]
and the right-hand side lies in the quadratic module generated by \(f\), hence in \(T(f)\). Since \(\sos\subseteq T(f)\), the result follows.
\end{proof}

\begin{corollary}\label{cor:algsetepsilon}
Let
\[
V=\{x\in\R^d: f_1(x)=\cdots=f_m(x)=0\}.
\]
If \(p\in\Rxd\) satisfies \(p\ge 0\) on \(V\), then for every \(\eps>0\) there exist \(\sigma\in\sos\) and \(h\in (f_1,\dots,f_m)\) such that
\[
p+\eps=\sigma+h.
\]
\end{corollary}

\begin{proof}
Apply Schm\"udgen's strict positivity theorem to the family \((f_1,-f_1,\dots,f_m,-f_m)\). The polynomial \(p+\eps\) is strictly positive on \(V\), hence belongs to the associated preordering. Proposition \ref{prop:algpreorder} then gives the desired decomposition.
\end{proof}

\begin{remark}
This says that, modulo the vanishing ideal of the variety, nonnegative polynomials can be approximated arbitrarily well by sums of squares. Positivity on an algebraic set is therefore positivity in the quotient ring up to an \(\eps\)-perturbation.
\end{remark}

\subsection{Gram matrices and a quantitative bound}

\begin{theorem}\label{thm:grambound}
Let \(p\in\Rxd\) be a sum of squares of degree \(2m\). Then \(p\) can be written as a sum of at most
\[
\binom{d+m}{m}
\]
squares of polynomials.
\end{theorem}

\begin{proof}
Let \(v_m(x)\) be the column vector of all monomials of degree at most \(m\); its length is
\[
N_m:=\binom{d+m}{m}.
\]
Since \(p\) has degree at most \(2m\), there exists a symmetric matrix \(G\in M_{N_m}(\R)\) such that
\[
p(x)=v_m(x)^{\mathsf T}Gv_m(x).
\]
If \(p=\sum_{k=1}^r h_k^2\), then each \(h_k\) has degree at most \(m\) and can be written as \(h_k(x)=c_k^{\mathsf T}v_m(x)\). Hence
\[
p(x)=v_m(x)^{\mathsf T}\Bigl(\sum_{k=1}^r c_kc_k^{\mathsf T}\Bigr)v_m(x),
\]
so \(p\) admits a positive semidefinite Gram matrix. Diagonalizing \(G\) gives
\[
G=U^{\mathsf T}\diag(\lambda_1,\dots,\lambda_{N_m})U
\]
with \(\lambda_j\ge 0\). Therefore
\[
p(x)=\sum_{j=1}^{N_m}\lambda_j\,\ell_j(x)^2,
\]
where the \(\ell_j\) have degree at most \(m\). Terms with \(\lambda_j=0\) vanish, so the number of nonzero squares is at most \(N_m\).
\end{proof}

\end{document}